\documentclass[12pt,draftcls,onecolumn]{IEEEtran}

\IEEEoverridecommandlockouts
\usepackage{isolatin1}
\usepackage{times}
\usepackage{pdfsync}
\usepackage{epsfig}
\usepackage{epstopdf}
\usepackage{amsmath}
\usepackage{amssymb}
\usepackage{mathrsfs}
\usepackage{color}
\usepackage{cite}
\usepackage{fix2col}

\newcommand{\eq}{\triangleq}

\newcommand{\field}[1]{\mathbb{#1}}
\newcommand{\R}{\field{R}}
\newcommand{\N}{\field{N}}
\newcommand{\U}{\field{U}}

\newcommand{\B}{\field{B}}

\newcommand{\Prob}{\mathbf{Pr}}
\newcommand{\E}{\mathbf{E}}
\newcommand{\K}{\mathcal{K}}

\newcommand{\hfs}{\hfill\ensuremath{\square}}

\newtheorem{thm}{Theorem}

\newtheorem{lem}{Lemma}

\newtheorem{assumption}{Assumption}

\graphicspath{{Matlab/}}

\title{Stochastic Stability of Event-triggered Anytime Control}

\author{Daniel E.~Quevedo,~\IEEEmembership{Member,~IEEE,}\thanks{D.\ Quevedo
    is with the School of Electrical Engineering \& Computer Science, The
    University of Newcastle, Australia, {\tt dquevedo@ieee.org}.  V.\ Gupta and
    W.\ J.\ Ma are with
    the Department of Electrical Engineering, University of Notre Dame, USA, {\tt
      vgupta2@nd.edu, wma1@nd.edu}. S.\ Y\"uksel
    is with the Department of Mathematics and Statistics, Queen's University,
    Kingston, ON K7L 3N6 Canada, {\tt yuksel@mast.queensu.ca}.  Research
    supported for the first author under Australian Research Council's Discovery Projects funding  scheme (project number DP0988601) and in part for the second and third authors  by NSF awards 0846631 and
    0834771.}  Vijay
  Gupta,~\IEEEmembership{Member,~IEEE,} Wann-Jiun Ma,~\IEEEmembership{Student
    Member,~IEEE,} Serdar Y\"uksel,~\IEEEmembership{Member,~IEEE}}

\begin{document}

\maketitle

\begin{abstract}
  We investigate control of a non-linear process when
  communication and  processing capabilities are limited. The sensor  communicates with a
  controller
  node
  through an erasure channel which introduces i.i.d.\ packet dropouts. Processor
  availability for control is random and, at times, insufficient to calculate
  plant inputs. To make efficient use of communication and processing resources,
  the sensor only transmits when the plant state lies outside a bounded target
  set. Control calculations are triggered by the received data. If
  a plant state measurement is successfully received and while
  the processor is available for control,  the  algorithm recursively calculates a sequence
  of tentative   plant
  inputs, which are stored in a buffer for potential future use. This safeguards
  for time-steps when the  processor is unavailable for control. We derive
  sufficient conditions on system parameters for stochastic stability of the
  closed loop and illustrate   performance gains   through numerical studies.
\end{abstract}


\linespread{1.3}

\section{Introduction}
Due to the tight coupling among the cyber and the physical cores in many cyber-physical systems, it is imperative to develop systematic design principles for controllers with limited communication and processing resources. Both the areas of control with communication constraints and control with limited and time-varying processor availability have accordingly received much attention.

\par Control design in the presence of practical communication channels and
protocols
has been studied in the area of networked control systems. 
Of particular interest to the present work is the literature on
control across analog erasure channels; 
see, e.g.,~\cite{gupdan09,imeyuk06,schsin07,quenes12a}.
Apart from arising from data transmission across a wireless channel, data loss may
also arise due to congestion in a communication network, possibly transmitted by a control loop. 
To minimize this source of data loss, one can conceive event
triggered communication schemes in which sensors transmit
information only if the system state exceeds a certain bound; see,
e.g.,~\cite{lilem10,tabuad07,xuhes04,ramsan11b,xiagup13}. Recently, work has also been done on designing event triggering rules
to ensure stability in the face of data dropouts. However,
most works are restricted to  single integrator dynamics~\cite{rabjoh09,bliall11} or are numerical studies~\cite{cerhen08}. 

\par On the other hand, various works have also considered the impact of
limited or time-varying processing power on closed-loop control~\cite{govfer99,henake04,andseu13a}. 
Interestingly, event-triggered and self-triggered updates of the control
inputs have also been proposed to ensure less demand on the processor
on average by calculating the control input on demand~\cite{cervel10,tabuad07}. The direction of anytime control has also
shown promise~\cite{bhabal04,grefon07,gupluo13,quegup13a}. Such algorithms calculate a coarse control input
even with limited processing resources and refine the input as more processing resources become
available. The quality of control inputs is thus time-varying, but no control
input is obtained only rarely. 

\begin{figure}[t]
  \centering
  \includegraphics[width=.8\textwidth]{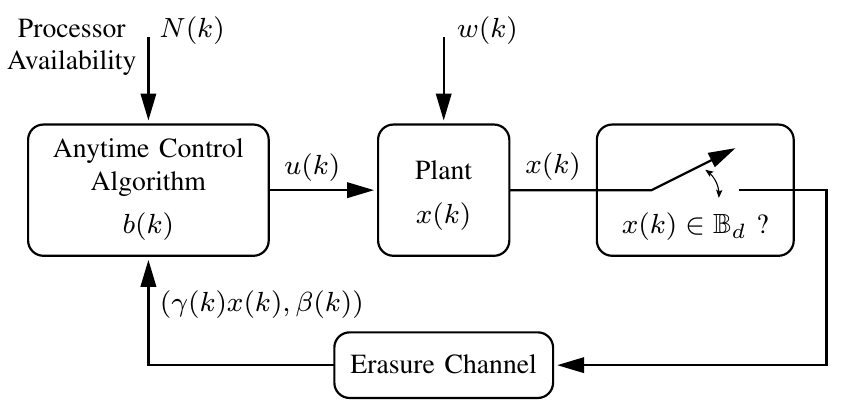}
  \caption{Event-triggered Anytime Control with Unreliable Communications.}
  \label{fig:scheme}
\end{figure}

\par Notwithstanding the advances discussed above, relatively few works have considered control design under both limited
communication and processing resources. Optimal control design for arbitrary non-linear processes under
communication and processing constraints is likely a challenging problem,
since certainty equivalence would not hold in general~\cite{ramsan11a}.
Accordingly, in the present note we consider a pre-designed control law, and
focus on the implementation of  this controller in the presence of both communication and
processing limitations.
As depicted in Fig.~\ref{fig:scheme}, we consider a discrete-time non-linear plant being controlled across a
communication
network that stochastically erases  data transmitted across it. To reduce
congestion in the network, the sensor employs an event triggered communication strategy. However, due to time-varying availability of the
processing resources, it is not guaranteed that the processor can calculate a
control input at all time steps when the sensor transmits (even if the network does
not erase the data). To maximally utilize the  processing resources, the
controller employs an anytime control algorithm. Under such a setting, we
analyze stochastic  stability of the closed loop. Our  main stability results
are stated
in terms of an
inequality that relates open-loop growth of the plant state, packet erasure
probability, and parameters of the processor availability model. For the particular
case where   processing resources are available at every time step,  our result reduces
to a sufficient condition for stochastic stability of non-linear control
in where sensor communicates according to an event-triggering
condition  across an analog i.i.d.\
erasure link. This may be of independent interest.



\paragraph*{Notation}
\label{sec:notation}
We write $\N$ for $\{1, 2,  \ldots\}$ and $\N_0$ for $\N \cup \{0\}$. $\R$ represents the real numbers and
$\R_{\geq 0}\eq [0,\infty)$. The $p\times p$ identity matrix is
denoted via $I_p$,  $0_{p \times q}$ is the $p\times q$  all-zeroes
matrix, $0_p\eq 0_{p\times p}$, and $\mathbf{0}_p\eq 0_{p\times 1}$.  The notation
$\{x\}_{\K}$ stands for
$\{x(k) \;\colon k \in \K\}$, where $\K\subseteq \N_0$.
We  adopt the conventions
$\sum_{k=\ell_1}^{\ell_2}a_k = 0$ and $\prod_{k=\ell_1}^{\ell_2}a_k = 1$,  if $\ell_1 > \ell_2$ and irrespective of
$a_k \in \R$. The superscript $^T$  refers to transpose. The
Euclidean norm of
a vector $x$ is denoted via $|x|=\sqrt{x^Tx}$.
 A
function $\varphi\colon \R_{\geq 0}\to \R_{\geq 0}$ is of
\emph{class-}$\mathscr{K}_\infty$ ($\varphi \in \mathscr{K}_\infty$), if it is
continuous, zero at zero, strictly increasing, and  unbounded. The probability of an event
$\Omega$ is denoted by
$\Prob\{\Omega \}$ and the conditional probability of $\Omega$ given
 $\Gamma$ by $\Prob\{\Omega\,|\,\Gamma \}$. The  expected value of a
 random variable $x$ given
 $\Gamma$ is denoted by  $\E\{x  \,|\, \Gamma \}$, while $\E\{x\}$ refers
 to the unconditional expectation.   The expression $x \sim \nu$ denotes that the random variable $x$ has probability distribution $\nu$ and $\E_{\nu}\{x\}$ denotes the expectation under probability distribution $\nu$.


\section{Event-driven Control over an Erasure Channel}
\label{sec:prob_form}

We consider non-linear (and possibly open-loop unstable) plants, sampled periodically with sampling interval
  $T_s>0$ (see Fig.~\ref{fig:scheme}),
\begin{eqnarray}
  \label{eq:process}
  x(k+1) = f(x(k),u(k)),\quad k\in\N_0,
\end{eqnarray}
where $x\in \R^n$ is the
   plant state, and $u\in \mathbb{U}\subseteq \R^p$ with $\mathbf{0}_p\in\U$ is the (possibly
   constrained) plant input. The initial state
 $x(0)$ is arbitrarily distributed. The plant is equipped with a sensor, which
has direct access to the plant state at the sampling instants $k\in\N_0$.
\par To save on
  communication expenditure, the sensor  adopts an
event-triggered transmission strategy, in which the sensor transmits only at
instances $ k\in \N_0$, where
$ x(k) \not \in \mathbb{B}_d\eq \{x\in\R^n \colon |x| < d\}$.
This transmission is across an erasure channel which introduces random packet
dropouts. 
To keep communication costs low,  the controller does  not send
acknowledgments back to the sensor and no re-transmissions are allowed.
We introduce  two discrete random processes, namely
$\{\gamma\}_{\N_0}$ and $\{\beta\}_{\N_0}$.
The binary transmission success  process
$\{\gamma\}_{\N_0}$ 
describes packet loss: a successful transmission at time $k$ is denoted by  $\gamma(k)=1$ and a packet erasure by $\gamma(k)=0$. The ternary process $\{\beta\}_{\N_0}$ incorporates the event-based transmission
rule, 
\begin{eqnarray}
  \label{eq:13}
  \beta(k)=
  \begin{cases}
    \gamma(k)&\text{if the sensor transmitted at time $k$,}\\
    2 &\text{if the sensor did not transmit at time $k$.}
  \end{cases}
\end{eqnarray}
Thus, $\beta(k)=2 \Leftrightarrow |x(k)|<d$.
We assume that $\beta(k)$ is known to the
controller at time $k$ through monitoring of received energy in the sensor transmission band.
 Transmission outcomes trigger  the functions carried out by the
controller. The scalar $d \in\R_{\geq 0}$ is a design parameter, which determines communication
channel utilization and control performance. Elucidating the trade-off between these quantities is
one of the motivations of the present work.

\par When implementing discrete-time control systems, it is generally  assumed that
the processing resources available to the
controller are such that the  desired
control law can be evaluated within
a fixed time-delay, say $\delta\in (0,T_s)$. 
However, in practical
networked and embedded systems, the processing resources
available for control calculations may vary and, at times, be insufficient to
generate a control input within the prescribed time-delay $\delta$ \cite{andseu13a}.
In the sequel we will further develop our anytime control
algorithm of \cite{quegup13a,quegup11a} to seek
favorable trade-offs between processor and communication availability, and
control performance. We will assume that the    plant model~(\ref{eq:process})
is globally stabilizable via state feedback.

\begin{assumption}[Stabilizability]
\label{ass:CLF}
There exist  $V\colon
\R^n\to\R_{\geq 0}$,  $\varphi_1, \varphi_2\in\mathscr{K}_\infty$,
  $\kappa \colon \R^n\to \mathbb{U}$, and a constant $\rho \in [0,1)$,
 such that
\begin{equation}
  \label{eq:3}
  \begin{split}
    \varphi_1(|x|)\leq V(x)\leq \varphi_2(|x|), \quad &\forall x\in\R^n,\\
     V(f(x,\kappa(x))) \leq \rho V(x), \quad
    &\forall x \notin \mathbb{B}_d.
  \end{split}
\end{equation}
\end{assumption}
\vspace{2mm}
To encompass processing constraints, we will assume that the controller needs
processor time to carry out
mathematical computations, such as evaluating $\kappa$. However, input-output
operations  and simple operations at a bit level, e.g., writing data into
buffers, shifting buffer contents and setting values to zero, do not require
processor time.

 \par Before proceeding we note that a direct implementation of $\kappa$  used in
Assumption~\ref{ass:CLF}, when processing
resources are time varying, sensor transmissions are event-triggered, and the
sensor transmissions are affected by dropouts,  results in the {\emph{baseline
    event-based  algorithm}}
\begin{equation}
  \label{eq:4}
  u(k) =
  \begin{cases}
    \kappa(x(k)) &\text{if  $\beta(k)=1$ and processor is available,}\\
    \mathbf{0}_p&\text{otherwise,}
  \end{cases}
\end{equation}
where the symbol
 $u(k)$ with $k\in\N_0$ denotes  the plant input which is applied during the
 interval
 $[kT_s +\delta, (k+1)T_s +\delta)$.
Whilst the baseline algorithm is  intuitive, our
previous works\cite{quegup11a,quegup13a} suggest that it will
be outperformed by
more elaborate control formulations.

\section{Event-driven Anytime Control Algorithm}
\label{sec:event-driven-anytime}
The anytime algorithm is based on the following idea:  control
calculations are triggered   whenever a new measurement is
successfully received. However, the precise number of control inputs calculated depends on the processing resources available.  At time intervals when the
controller is provided with more processing resources than are needed to evaluate
the current control input, the algorithm  calculates a  sequence of tentative
future plant  inputs.  The sequence is stored in a local
buffer and may be  used
when, at some future time
steps, the  processor
availability precludes any control calculations even though new state information is received. 

\par In our recent work\cite{quegup13a,quegup11a}, we analyzed this algorithm
for the simpler
case where the controller has direct access to plant state $x(k)$ at all
instants $k\in \N_0$. In the present work we alleviate this assumption by considering that
sensor transmissions are event-triggered and through a communication
channel which introduces random dropouts. In addition, to save energy and
processing resources, the controller is event-triggered.  More precisely,
the actions taken by the controller are guided by the value of $\beta(k)$ and
the processor availability.

\par If $\beta(k)=1$,
then the controller  uses $x(k)$ to calculate tentative control values, provided the
processor is available for control. This sequence will be stored in a buffer. If
the processor is not available or $\beta(k)=0$, then the controller does
not do any calculations and the plant
input is provided by previously calculated buffered values (if available). The instances
$\beta(k)=2$ refer to  situations where the plant state is at the desired region
$\mathbb{B}_d$, and $x(k)$ is not sent to the controller.  In this scenario, the plant input
is set to zero, the
buffer is emptied, and the controller is switched off until the system state moves
out of the desired region $\mathbb{B}_d$ and a new state measurement is
received.  Fig.~\ref{fig:triggering} outlines the proposed algorithm. In this
figure,
\begin{equation*}
     S\eq
\begin{bmatrix}
    0_p & I_p& 0_p &\hdotsfor{1} &0_p\\
    \vdots & \ddots & \ddots &\ddots  & \vdots\\
    0_p & \hdotsfor{2}       &  0_p & I_p\\
 0_p &\hdotsfor{3} &  0_p
\end{bmatrix}\in\R^{\Lambda  p\times \Lambda p},\quad
  b(k) =
  \begin{bmatrix}
    b_1(k)\\b_2(k)\\\vdots \\ b_{\Lambda}(k)
  \end{bmatrix},
\end{equation*}
where $\{b\}_{\N_0}$ denote the buffer states for a given buffer size $\Lambda\in\N$ and each $b_j(k)\in\R^p$, $j\in\{1,\dots,\Lambda\}$.

\begin{figure}[t]
  \centering
  \includegraphics[width=\textwidth]{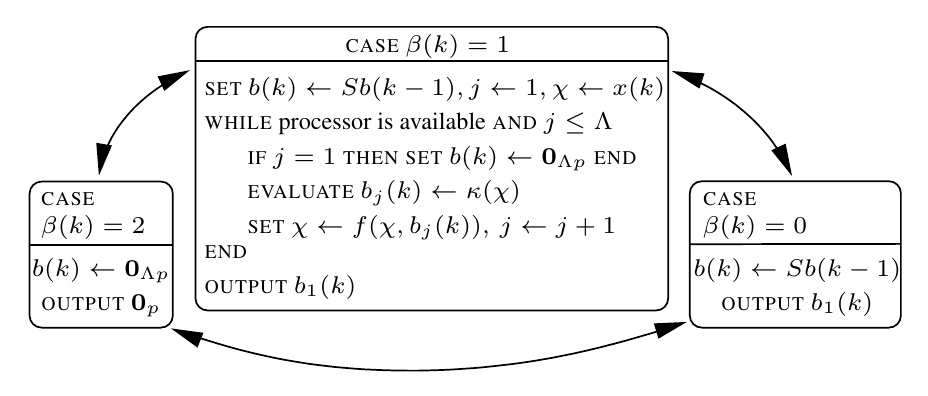}
  \caption{Operating modes of anytime Algorithm A$_1$ during the time interval
    $[kT_s, (k+1)T_s)$.}
  \label{fig:triggering}
\end{figure}

\linespread{1.25}
\begin{figure}[h!]
\noindent\rule{\linewidth}{0.2mm}\vspace{-2mm}
\begin{enumerate}[\setlabelwidth{Step 5}]
\item[Step 1:]  At time $t=0$,
  \par\hspace{1cm} \textsc{set}  $b(-1)\leftarrow \mathbf{0}_{\Lambda  p}$, $k\leftarrow 0$
\item[Step 2:] \label{step:timek}
  \textsc{if}  $t \geq k T_s$, \\
   \textsc{then}
   \par\hspace{1cm} \textsc{switch} $\beta(k)$
  \par\hspace{2cm} \textsc{case} $2$,
  \par\hspace{3cm} \textsc{set} $b(k)\leftarrow\mathbf{0}_{\Lambda p}$, $j\leftarrow 1$;
  \par\hspace{3cm} \textsc{goto} Step 4;
   \par\hspace{2cm} \textsc{case} $0$,
     \par\hspace{3cm}  \textsc{set}  $j\leftarrow 1$,   $b(k)\leftarrow Sb(k-1)$;
  \par\hspace{3cm} \textsc{goto} Step 4;
  \par\hspace{2cm} \textsc{otherwise}
  \par\hspace{3cm}  \textsc{input} $x(k)$;
  \par\hspace{3cm}  \textsc{set}  $\chi\leftarrow x(k)$, $j\leftarrow 1$,
   $b(k)\leftarrow Sb(k-1)$;
   \par\hspace{1cm} \textsc{end}\\
  \textsc{end}

%
%
%
%

\item[Step 3:] \label{step:repeat2}
  \textsc{while} ``sufficient processor time is available'' and  $j\leq \Lambda$ and
  time $t < (k+1) T_s$,

  \par\hspace{1cm} {\textsc{evaluate} $u_j(k)=\kappa(\chi)$;}

  \par\hspace{1cm} \textsc{if} $j=1$, \textsc{then}
  \par\hspace{2cm} \textsc{output} $u_{1}(k)$;
   \par\hspace{2cm} \textsc{set} $b(k)\leftarrow\mathbf{0}_{\Lambda p}$;
  \par\hspace{1cm} \textsc{end}

  \par\hspace{1cm} \textsc{set} $b_j(k)\leftarrow u_j(k)$;

  \par\hspace{1cm}  \textsc{if} ``sufficient processor time is not available'' or
  $t \geq (k+1) T_s$, \textsc{then}
  \par\hspace{2cm} \textsc{goto} Step 5;
  \par\hspace{1cm} \textsc{end}

  \par\hspace{1cm} \textsc{set} $ \chi \leftarrow  f(\chi,u_j(k))$,
  $j\leftarrow j+1$;\\
  \textsc{end}

\item[Step 4:]
  \textsc{if} $j=1$, \textsc{then}
  \par\hspace{1cm} \textsc{output} $b_{1}(k)$;\\
  \textsc{end}

\item[Step 5:]
  \textsc{set} $k \leftarrow k+1$ and \textsc{goto} Step 2;
\end{enumerate}
\vspace{-3mm}
\noindent\rule{\linewidth}{0.2mm}
\caption{Algorithm A$_1$}
\label{alg:1}
\end{figure}
\linespread{1.3}

\par For future use, we will
denote by
$N(k)\in\{0,1,\dots,\Lambda\}$ the total number of iterations of the while-loop which are carried out during the
interval $t\in [kT_s,(k+1)T_s)$. Thus, as described above,
if $N(k)\geq 1$,
then the entire sequence of
tentative controls is
$    \{b_{1}(k), b_2(k),  \dots ,  b_{N(k)} (k)\} $
and the plant input is set to $b_{1}(k)$.
 If $N(k)=0$, then  the plant input depends on the variable $\beta(k)$. If
$\beta(k)\in\{0,1\}$ (i.e., $x(k)$  does not lie inside the
desired region), then  $u(k)$ is taken as the
first $p$ elements of the shifted state $b(k)=Sb(k-1)$. If, on the other hand,
$\beta(k)=2$ indicating that  $x(k)\in\mathbb{B}_d$, then the buffer is emptied and the plant input is
set to zero, see Fig.~\ref{fig:triggering}.

\par Algorithm A$_1$    amounts to a dynamic state
feedback policy with internal state variable $b(k)$ which provides the
plant input
 $ u(k)$
 and suggested plant
inputs at future time steps. If new state information is received and more processor time is
available, a longer trajectory of  control inputs is calculated and
stored in the buffer. If the buffer runs out of tentative plant inputs, then
actuator values are set to zero.   The algorithm does not require prior knowledge of future
 processor  availability and hence can be employed  in shared systems where the controller
task can be preempted by other computational tasks at the processor.

\section{Stochastic Stability - Preliminaries}
\label{sec:analysis}


 For our subsequent analysis, 
 it is convenient to investigate how many
values in the  state $b(k)$ stem from evaluating $\kappa$,
$\ell \in\N_0$. As in\cite{quegup13a,quegup11a}, we will refer to this value as the \emph{effective buffer
  length} (at time $k$), and denote it as
 $ \lambda (k) \in\{0,1,\dots,\Lambda\},  k\in\N_0$
with $\lambda(-1)=0$. It is easy to see that for all $k\in \N_0$
we have
\begin{equation*}
  \lambda(k) =
  \begin{cases}
  N(k)&\text{if $N(k)\geq 1$,}\\
 \max\{0,\lambda(k-1)-1\},&\text{if $N(k)=0$ and $\beta(k)\in\{0,1\}$,}\\
  0& \text{if $\beta(k)=2$.}
  \end{cases}
\end{equation*}
To investigate  stability, we make the following assumptions:
\begin{assumption}[Processor availability]
\label{ass:iid}
The sampling time of the plant~(\ref{eq:process}) is such that processor
availability for control at different time-instants is  independent and
  identically distributed (i.i.d.).  Thus, the process $\{N\}_{\N_0}$ has conditional probability distribution
$  p_{j}\eq  \Prob \{N(k)=j\,|\,\beta(k)=1\} ,$
  where $p_j\in[0,1)$ are given and with $j\in\{0,1,2,\dots, \Lambda\}$. For other realizations of $\beta(k)$, no
   plant inputs are calculated, thus,
 $     \Prob \{N(k)=0\,|\,\beta(k)\in\{0,2\}\} = 1$.
\hfs
\end{assumption}
\begin{assumption}[Erasure channel]
  \label{ass:dropouts}
  The binary transmission success  process $\{\gamma\}_{\N_0}$ has conditional
  probabilities
   $\Prob\{\gamma(k)=1\,|\,|x(k)|\geq d\}=q$, $\Prob\{\gamma(k)=0\,|\,|x(k)|<
   d\}=1$.\hfs
\end{assumption}
%
\begin{assumption}[Open-loop bound]
  \label{ass:bound_prob}
 There exists $\alpha\geq \rho$ 
such that
  \begin{eqnarray}
    \label{eq:20}
    V({f}(\chi,\mathbf{0}_p))\leq\alpha V(\chi),\quad\forall \chi
    \in\R^n.
  \end{eqnarray}
where  $\rho, V$ and $ \varphi_2$ are as in~(\ref{eq:3}). Further,
$ \E\big\{\varphi_2(|x(0)|)\big\}<\infty$.\hfs
\end{assumption}
It is worth noting that,  by allowing for $\alpha >1$,  Assumption~\ref{ass:bound_prob} does not require that
  the open-loop system $x(k+1) = {f}(x(k),\mathbf{0}_p)$  be
  asymptotically stable. 
  Further, note that
  Assumptions~\ref{ass:CLF} and~\ref{ass:bound_prob} are stated in terms of the
  same
   function $V$, see  also\cite[Section IV-A]{quegup13a}.

\par To go beyond stability and investigate   stationarity, it is convenient  to
impose the following assumptions on the control policy  $\kappa$
\begin{assumption}[Continuity of $\kappa$]
\label{ass:Continuity}
The control law $\kappa$ in~(\ref{eq:3}) is such that $\kappa(x)=\mathbf{0}_n$ for all $x \in
\mathbb{B}_d$ and $\kappa$ is continuous on $\mathbb{R}^n$. \hfs
\end{assumption}



\section{Stability with the Baseline Algorithm}
\label{sec:stab-with-basel}
 If the baseline
algorithm is used and Assumption~\ref{ass:iid} holds, then 
\begin{eqnarray}
  \label{eq:10}
  x({k+1}) =
  \begin{cases}
   f(x(k),\kappa(x(k))), &\text{if $N(k)\geq 1$,}\\
   f(x(k),\mathbf{0}_p),&\text{if $N(k)=0$}.
  \end{cases}
\end{eqnarray}
The following result establishes  conditions on system parameters which ensure
that the closed loop~(\ref{eq:10}) is stable in a stochastic sense.

\begin{thm}[Stability with  baseline algorithm]
\label{thm:baseline}
Consider~(\ref{eq:10}) and define
 $ D\eq \varphi_2(d)$.
 Suppose that Assumptions~\ref{ass:CLF} to~\ref{ass:bound_prob} hold and that
  \begin{eqnarray}
    \label{eq:baseline_stability}
    \Gamma \eq (1-q)\alpha +q \big(p_{0}\alpha+(1-p_{0})\rho\big) < 1,
  \end{eqnarray}
where $\rho\in[0,1)$ is the closed-loop bound in~(\ref{eq:3}), $\alpha$
is the  bound in~(\ref{eq:20}),
 $q$ is the transmission
success probability, and $p_0$ is the probability of the processor not being
available for control.
Then for all $x\in\N_0$,
\begin{equation*}
    \E\big\{\varphi_1(|x(k)|)\big\}
 \leq
    \Gamma^k \E\big\{\varphi_2(x(0)) \big\} + \frac{q (1-p_{0})
      (\alpha-\rho)D}{1-\Gamma}<\infty.
\end{equation*}
\end{thm}
\begin{proof}
  Note that, for i.i.d.\ processor and channel availabilities  $\{x\}_{\N_0}$ in~(\ref{eq:10})
  is Markovian. This can be verified by
  noting that conditioning on $x(k)$ makes the event
  outcome $\beta(k)$ depend on $\gamma(k)$ only. To analyze stochastic
  stability  using   Lyapunov functions (see, e.g.,\cite{meyn89}),
we use the law of total
  expectation to write
  \begin{gather}
  \notag \E\big\{V(x(1))\,\big|\,x(0)=\chi\big\}
      =\sum_{j=0}^2\E\big\{V(x(1))\,\big|\,x(0)=\chi,\beta(0)=j\big\} 
      \label{eq:6} 
      \Prob\{\beta(0)=j\,|\,x(0)=\chi\}.
    \end{gather}
    If we now use (\ref{eq:13}),~(\ref{eq:3}),~(\ref{eq:20}) and the definition of $\mathbb{B}_{d}$, then:
\begin{equation}
  \label{eq:8}
  \begin{split}
    \E\big\{V(x(1))\,\big|\,x(0)=\chi,\beta(0)=0\big\}&\leq \alpha V(\chi)\\
    \E\big\{V(x(1))\,\big|\,x(0)=\chi,\beta(0)=2\big\}&\leq \alpha
    V(\chi)<\alpha  \varphi_2(d),
  \end{split}
\end{equation}
For $\beta(0)=1$,   $x(0)$ is received. Using~(\ref{eq:20}) and~(\ref{eq:10}), we have
\begin{align}
\nonumber    \E&\big\{V(x(1)) \,\big| \, x(0)=\chi,\beta(0)=1 \big\} =\sum_{j\in\N_0}
    \E\big\{V(x(1))\,\big|\,x(0)=\chi,\beta(0)=1,N(0)=j\big\}\\
    & 
    \qquad \times \Prob\{N(0)=j\,|\,x(0)=\chi,\beta(0)=1\} 
   \leq \big(p_{0}\alpha+(1-p_{0})\rho\big)V(\chi).
     \label{eq:14}
\end{align}
Now, if $x(0)\in\mathbb{B}_d$, then
$\beta(0)=2$, thus~(\ref{eq:6}) and~(\ref{eq:8}) provide
\begin{eqnarray}
  \label{eq:17}
  \E\big\{V(x(1)) \,\big| \, x(0)=\chi\in\mathbb{B}_d \big\}\leq \alpha V(\chi).
\end{eqnarray}
Further, since $\alpha
-\Gamma= q(1-p_0)(\alpha-\rho)>0$ (see~(\ref{eq:baseline_stability}))  and
$V(\chi)<  D$ for all $\chi\in\mathbb{B}_d$,  we have  
\begin{equation}
  \label{eq:26}
  (\alpha-\Gamma) V(\chi) <  (\alpha-\Gamma) D\Rightarrow \alpha V(\chi) <
  \Gamma  V(\chi) + (\alpha-\Gamma) D,\quad \forall \chi \in\B_d.
\end{equation}
 On the other hand, if $x(0)\not\in\mathbb{B}_d$, then (in view of Assumption~\ref{ass:dropouts}), $\Prob\{\beta(0)=0\,|\,x(0)\not\in\mathbb{B}_d\}=1-q$,
and $\Prob\{\beta(0)=1\,|\,x(0)\not\in\mathbb{B}_d\}=q$. Thereby, substitution of~(\ref{eq:8})
and~(\ref{eq:14}) into~(\ref{eq:6}) provides:
\begin{eqnarray}
  \label{eq:18}
  \E\big\{V(x(1)) \,\big| \, x(0)=\chi\not\in\mathbb{B}_d \big\}\leq \Gamma V(\chi).
\end{eqnarray}
Expressions~(\ref{eq:17})--(\ref{eq:18}) lead to:
\begin{equation*}
     \E\big\{V(x(1))\,\big|\,x(0)=\chi\big\}<
    \Gamma V(\chi)+ ( \alpha - \Gamma )D=  \Gamma V(\chi)+q (1-p_{0}) (\alpha-\rho)D.
\end{equation*}
Consequently, Proposition 3.2 of \cite{meyn89}, and~(\ref{eq:3}) give
\begin{equation*}
  \E\big\{\varphi_1(|x(k)|)\,|\,x(0)=\chi\big\} \leq   \Gamma^k V(\chi) +
  \frac{q (1-p_{0}) (\alpha-\rho)D}{1-\Gamma} ,
\end{equation*}
 for all $k\in
  \N_0$. Using the law of total expectation and~(\ref{eq:3}) yields the first
inequality. The second follows from Assumption~\ref{ass:bound_prob}.
\end{proof}

It is worth noting that 
whilst the condition~(\ref{eq:baseline_stability}) is independent of the size
of   $\mathbb{B}_{d}$, the ultimate bound is
increasing in $d$. We can also consider two special cases. If $d=0$ and  $q=1$, so that the sensor transmits at every instant $k\in\N_0$ and the
communication channel
does not introduce any dropouts, (\ref{eq:baseline_stability}) reduces to
$p_{0}\alpha+(1-p_{0})\rho<1$, thus recovering our earlier
result\cite[Thm.1]{quegup13a}.
If the processor is available at
every time-step (i.e.,  $p_0=0$), then the situation amounts to
event-based control for non-linear systems using an erasure channel. In this
case, the sufficient condition~(\ref{eq:baseline_stability}) becomes
$(1-q)\alpha+\rho q<1$.

%


\begin{thm}[Stationarity with  baseline algorithm]\label{StocBaseLine}
Consider~(\ref{eq:10}), suppose that
  Assumptions~\ref{ass:CLF} to~\ref{ass:Continuity} hold and
  that~(\ref{eq:baseline_stability}) holds. Then, there exists an
  invariant probability measure for $\{x\}_{\N_0}$. Furthermore, under every such
  invariant probability measure  $\pi$, $$\E_{\pi}\{\varphi_1(|x|) \} \leq
q (1-p_{0})
      (\alpha-\rho) \varphi_2(d)/(1-\Gamma).$$ 
\end{thm}
\begin{proof}
   Let ${\cal P}(\mathbb{R}^n)$ denote the set of probability measures on $\mathbb{R}^n$ and define for every Borel $B$, $v_T(B) = (1/T)\E\{\sum_{k=0}^{T-1} 1_{\{x(k) \in B\}}\},$
such that $v_T \in {\cal P}(\mathbb{R}^n)$ forms an expected empirical occupation measure sequence. We then have,
$$\langle v_T, \varphi_1 \rangle \eq \int v_T(dx) \varphi_1(|x|) =\frac{1}{T}\E\bigg\{\sum_{k=0}^{T-1} \varphi_1(|x(k)|) \bigg\}.$$
Let $t_0 \in \mathbb{N}$. By Theorem~\ref{thm:baseline}, we have that $\E\{\varphi_1(|x(k)|)\}$ and the subsequence
$\{\langle v_T, \varphi_1 \rangle, T \geq t_0\}$ are uniformly bounded by some $M_{t_0} < \infty$.
Define $N_r := \{x: \varphi_1(|x|) \leq r\}$. Since $\varphi_1$ is monotone and
unbounded, by an application of Markov's inequality, we have
\begin{eqnarray}\label{uniformBoundTight}
M_{t_0} \geq \int v_T(dx) \varphi_1(|x|) \nonumber 
\geq \int_{\mathbb{X} \setminus N_r}\! v_T(dx) \varphi_1(|x|) \geq r v_T(\mathbb{R}^n \!\setminus\! N_r).
\end{eqnarray}
Thus,
$v_T(N_r) \geq 1 - {M_{t_0}/r},$
and hence for every $\epsilon= {M_{t_0} / r} > 0$, there exists a compact set
$N_{{M_{t_0} / \epsilon}}=\{x: \varphi_1(|x|) \leq {M_{t_0} / \epsilon}\}$ such that $v_t(N_{{M_{t_0} / \epsilon}}) \geq 1 -
\epsilon$. The sequence $\{v_t, t \geq t_0\}$ is, hence, a tight sequence with a converging
subsequence $v_{t_k}$ converging to some $v^* \in {\cal P}(\mathbb{R}^n)$. By~(\ref{eq:10}), if $x(t) \in \mathbb{B}_d$ the control action
is zero and outside $\mathbb{B}_d$, either zero control is applied or
$\kappa(x(t))$ is applied. Since $\kappa$ is continuous and is zero inside
$\mathbb{B}_d$ (see Assumption~\ref{ass:Continuity}), the
Markov chain is weak Feller.\footnote{A Markov chain $\{x(k)\}_{k\in\N_0}$ is
  (weak) Feller if  $\E\{h(x(k+1))|x(k)=\chi\}$ is continuous in $\chi$, for every continuous and bounded function $h$.} Consequently, it can be shown that every limit of
such a subsequence is invariant (see, e.g.,   \cite[Ch.\ 12]{meytwe09}) and satisfies $\langle v_T, \varphi_1 \rangle \leq M_{t_0}$. By Theorem~\ref{thm:baseline}, by increasing $t_0$, $M_{t_0}$ can be taken to be arbitrarily close to $q (1-p_{0}) (\alpha-\rho) \varphi_2(d)/(1-\Gamma)$.
\end{proof}

\section{Stability with the Anytime Algorithm}
\label{sec:algorithm-a_1}
The analysis of the event-based anytime algorithm is more involved
than that of the baseline system~(\ref{eq:10}). First, due to buffering,
$\{x\}_{\N_0}$ will in general not be a Markov process. Further, the  distribution of $\{\beta\}_{\N_0}$ is difficult to derive for
general plant models. This makes  the  approaches of \cite{quegup11a,quegup13a} insufficient to treat the present
case.

\par  For ease of exposition,  we assume that the initial  effective
buffer length, $\lambda(0)=0$, and denote  the time steps where  $\lambda(k)=0$ via
$\mathcal{K}=\{k_i\}_{i\in\N_0}$, where $k_{0}=0$ and
 $ k_{i+1} = \inf \big\{ k\in\N \colon k>k_i,\;  \lambda(k)=0\big\}$, $i\in\N_0.$
 We also describe the amount of time steps between  consecutive
elements  of $\mathcal{K}$ via the process $\{\Delta_i\}_{i\in\N_0}$, where
$  \Delta_i\eq k_{i+1}-k_i$.
It is easy
to see that
\begin{eqnarray}
  \label{eq:22}
  \beta(k_i+\ell)\in\{0,1\},\quad \forall \ell \in \{1,2,\dots,\Delta_i-1\}, \quad
  \forall i\in\N_0
\end{eqnarray}
whereas $ \beta(k_i )\in\{0,1,2\}$, $\forall i\in\N_0$ and   $x(k^*)\in\mathbb{B}_d
\Rightarrow k^* \in\K$. In contrast to the cases
examined in \cite{quegup11a,quegup13a}, due to the
event-triggering mechanism, $\{\Delta_i\}_{i\in\N_0}$ is, in general, not
i.i.d. In fact, the distribution of $\Delta_i$ depends on $x(k_i)$ and is difficult to
characterize. To study stability of the event-based anytime algorithm, we will develop a state-dependent random-time drift condition.
Our first result,  states that whilst $\{x\}_{\N_0}$
is in general not  Markovian, the   state sequence
\emph{at the time steps
    $k_i\in\K$},    is a  Markov process.
\begin{lem}[Markov property of the sampled process]
\label{lem:Markov}
  Con\-si\-der~(\ref{eq:process})  controlled via  Algorithm A$_1$
  and  suppose that Assumptions~\ref{ass:iid} and~\ref{ass:dropouts} hold. Then
 $\{x\}_{\K}$
is  Markovian.\hfs
\end{lem}
\begin{proof}
   The definition of $\K$ gives that $\forall k_i\in\K$ we have
  $  u(k_i) =\mathbf{0}_p,$
  $ b(k_i)=\mathbf{0}_{\Lambda p}$,
  $\lambda(k_i) =  N(k_i)=0$. Thus, the plant state at
  time $k_{i+1}$  depends only on $x(k_i)$ and the sample paths
  $\{N(k_i+1),N(k_i+2),\dots,N(k_{i+1}-1)\}$ and
  $\{\gamma(k_i+1),\gamma(k_i+2),\dots,\gamma(k_{i+1}-1)\}$. The result follows
  since $\{N\}_{\N_0}$ and $\{\gamma\}_{\N_0}$ are i.i.d.
\end{proof}
 The following result provides a sufficient condition for stochastic
stability of the closed loop when the event-based anytime control algorithm  of
Section~\ref{sec:event-driven-anytime} is
used over an erasure channel.
\begin{thm}[Stability with   Algorithm A$_1$]
  \label{thm:anytime}
 Suppose that
  Assumptions~\ref{ass:CLF} to~\ref{ass:bound_prob} hold and define
\begin{eqnarray}
  \label{eq:19}
  \Omega\eq  \alpha\sum_{j\in\N}  \rho^{j-1}
     \Prob\{\Delta_i=j\,|\, \beta(k_{i+1})\not=2\}.
\end{eqnarray}
If Algorithm A$_1$  is used and $\Omega <1$, then
\begin{equation}
  \label{eq:25}
  \begin{split}
    &\max_{k\in\{k_i,k_i+1,\dots,k_{i+1}-1\}}
    \E\big\{\varphi_1(|x(k)|) \big\}  \leq \frac{1+\alpha -\rho}{1-\rho} \Omega^i \E\big\{\varphi_2(x(0)) \big\}
    +\frac{\varphi_{2}(d)}{1-\Omega}<\infty,\quad \forall i\in\N.
  \end{split}
\end{equation}
\end{thm}

\begin{proof}
We first note that for all $k_i\in\mathcal{K}$ and $\ell \in
\{1,\dots,\Delta_i-1\}$,
  $u(k_i) = \mathbf{0}_p$ and $u(k_i+\ell) = \kappa(x(k_i+\ell))$.
Therefore, the function $V(x(k_{i+1}))$ can be bounded by
using~(\ref{eq:3}) and~(\ref{eq:20}), leading to
\begin{eqnarray}
  \label{eq:17a}
   \E\{V(x(k_{i+1}))\,|\, x(k_i)=\chi,\Delta_i=j\} \leq
   \alpha \rho^{j-1}V(\chi), \forall \chi \in\R^n.
\end{eqnarray}
To account for   event-based transmission, we consider instances
where the
buffer is emptied triggered by $\beta(k)=2$. At
these
instances,~(\ref{eq:17a})   holds; further, $V(k_{i+1})< D\eq
\varphi_{2}(d)$. Thus, 
\begin{eqnarray}
  \label{eq:5}
  \E\{V(x(k_{i+1}))\,|\, x(k_i)=\chi,\Delta_i=j,\beta(k_{i+1})=2\} <
   D,\quad  \forall j\in\N.
\end{eqnarray}
By  using the law of
total expectation twice, we thus
obtain,
\begin{align}
\nonumber     &\E\{V(x(k_{i+1}))\,|\, x(k_i)=\chi\} = \E\{V(x(k_{i+1}))\,|\,
x(k_i)=\chi,\beta(k_{i+1})=2\}  
     \Prob\{\beta(k_{i+1})=2\,|\,x(k_i)=\chi\}\\
\nonumber     &\;+\E\{V(x(k_{i+1}))\,|\, x(k_i)=\chi,\beta(k_{i+1})\not=2\} 
     \Prob\{\beta(k_{i+1})\not=2\,|\,x(k_i)=\chi\}\\
\label{eq:21d}     &\leq D + \E\{V(x(k_{i+1}))\,|\,
x(k_i)=\chi,\beta(k_{i+1})\not=2\}\\
&\nonumber 
= D + \sum_{j\in\N} \E\{V(x(k_{i+1}))\,|\,
     x(k_i)=\chi,\beta(k_{i+1})\not=2,\Delta_i=j\}
     \Prob\{\Delta_i=j\,|\, x(k_i)=\chi,\beta(k_{i+1})\not=2\}\\
     &\nonumber \leq D+\sum_{j\in\N} \alpha \rho^{j-1}V(\chi)
     \Prob\{\Delta_i=j\,|\, x(k_i)=\chi,\beta(k_{i+1})\not=2\} = D+\Omega V(\chi),
     \forall \chi \in\R^n,
\end{align}
with $\Omega$ as in~(\ref{eq:19}) and where, to derive the last equality, we
have used Assumption~\ref{ass:iid}. Since $\{x\}_{\K}$ is Markovian,
\cite[Prop.\ 3.2]{meyn89} yields that $\Omega < 1$ guarantees
\begin{equation*}
  \E\big\{V(x(k_i))\,|\,x(k_0)=\chi\big\} \leq
  \Omega^i V(\chi) + \frac{D}{1-\Omega} ,\quad \forall i\in
  \N_0.
\end{equation*}
Now, since~(\ref{eq:17a}) holds, by a method similar to the one used in
the proof of\cite[Thm.1]{quegup11a}, we can establish the
 (admittedly loose) bound:
\begin{equation}
\label{driftCriterion2}
\begin{split}
  \E \Bigg\{ \sum_{k=k_i}^{k_{i+1}-1} &V(x(k))\,\bigg|\,x(k_0)=\chi\Bigg\}  \leq
  \frac{1+\alpha -\rho}{1-\rho} \Omega^i V(\chi) +\frac{D}{1-\Omega},\quad
  \forall i\in\N.
\end{split}
\end{equation}
Using the law total
expectation,~(\ref{eq:3}) and Assumption~\ref{ass:bound_prob}
gives~(\ref{eq:25}).
\end{proof}
The above result establishes a sufficient condition for the
system to be stochastically stable. The quantity~(\ref{eq:19}) is stated in
terms of a conditional
distribution of $\Delta_i$, which can be
characterized as follows:
\begin{lem}[Conditional distribution of $\Delta_i$]
\label{lem:delta}
 Suppose that Assum\-ptions~\ref{ass:iid} and \ref{ass:dropouts} hold and that
 Algorithm A$_1$ is used. We then have
\begin{equation}\label{wann}
     \frac{\Prob\{\Delta_i=j\,|\,\beta(k_{i+1})\not = 2  \}}{1-q+p_0q}
     =
     \begin{cases}
       1&\text{if $j=1$,}\\
       \theta^T
       {G}^{j-2}
       e_1&\text{if $j\geq 2$,}
     \end{cases}
\quad \forall (i,j) \in\N_0\times \N,
   \end{equation}
 where
$ \theta^T =
    q\begin{bmatrix}
      p_1&\dots&p_{\Lambda}
    \end{bmatrix}$ and $e_1^T=
      \begin{bmatrix}
     1&0&\dots&0
    \end{bmatrix}$.  In~(\ref{wann}), the   entries of the matrix
 ${G}=[g_{\ell j}]$, $\ell,j\in\{1,2,\dots,\Lambda\}$
 are  $g_{\ell j}=p_jq$, $\forall (\ell,j)\in\{3,4,\dots,\Lambda\}\times\{1,2,\dots,{\ell-2}\}\cup
     \{1,2,\dots,\Lambda\}\times\{\ell, \ell+1,\dots,\Lambda\}$; and
$g_{\ell(\ell-1)}=1-q+(p_0+p_{\ell-1})q$, $\forall \ell\in
\{2,3,\dots,\Lambda\}.$ \hfs
\end{lem}
\begin{proof}
  We first note that our focus is on the time sequences of the form
$\mathcal{I}_i\eq\{k_{i}+1,\dots, k_{i+1}\}$ where $k_i\in\mathcal{K}$, $i\in\N_0$ and where
$\beta(k)\not = 2$, $\forall k\in \mathcal{I}_i$.
Given Assumptions~\ref{ass:iid} and~\ref{ass:dropouts} and the buffering
mechanism described in Section~\ref{sec:event-driven-anytime}, it follows that
$\{\lambda(k)\}$ during every interval $k\in\mathcal{I}_i$,  $i\in\N_0$, is a
homogeneous Markov Chain.  The process $\Delta_i$ then  amounts to the first
    return times to $0$ of  this finite Markov Chain. To characterize the
    latter, we need to evaluate the
    transition probabilities
    $g_{\ell j}\eq\Prob\{\lambda(k+1)=j\,|\,\lambda(k)=\ell,
    k\in\mathcal{I}_i,k+1\in\mathcal{I}_i\}$. Without loss of generality, we
    will set $k=0$. We begin by considering transitions from $\ell\in \{0,1\}$
    to $0$:
\begin{equation*}
     \begin{split}
     g_{\ell0}&=\Prob\{N(1)=0\,|\beta(1)=0\}\Prob\{\beta(1)=0\,|\,\beta(1)\not
     =2\}\\
     &\quad+\Prob\{N(1)=0\,|\beta(1)=1\}\Prob\{\beta(1)=1\,|\,\beta(1)\not =2\}
=(1-q)+p_{0}q, \quad \forall \ell\in \{0,1\}.
     \end{split}
   \end{equation*}
For $\ell\in \{2,3,\dots,\Lambda\}$, we have $g_{\ell 0}=0$. The
 buffer length diminishes
by one for the scenarios  considered below:
\begin{equation*}
     \begin{split}
     g_{\ell(\ell-1)}
     &=\Prob\{N( 1)=0\,|\beta( 1)=0\}\Prob\{\beta( 1)=0\,|\,\beta( 1)\not =2\}
     +\Prob\{N( 1)=0\,|\beta( 1)=1\}\\
     &\quad \times\Prob\{\beta( 1)=1\,|\,\beta( 1)\not =2\}+\Prob\{N( 1)=\ell-1\,|\beta( 1)=1\}\Prob\{\beta( 1)=1\,|\,\beta(
     1)\not =2\}\\
     &=(1-q)+p_0 q+ p_{\ell-1}q, \quad \forall \ell\in \{2,3,\dots,\Lambda\}.
     \end{split}
   \end{equation*}
The  other   transitions    are related to when
$\lambda(k+1)=N(k+1)$, for
 $(\ell,j)\in\big\{\{3,4,\dots,\Lambda\}\times\{1,2,\dots,{\ell-2}\}\big\} \cup
 \big\{ \{1,2,\dots,\Lambda\}\times\{\ell,\ell+1,\dots,\Lambda\}\big\} \cup \big\{ 0 \times
 \{1,2,\dots,\Lambda\}\big\}$. Here we have:
\begin{equation*}
     \begin{split}
     g_{\ell j}&=\Prob\{\lambda( 1)=j\,|\lambda(0)=\ell,\beta( 1)=0\}\Prob\{\beta(
     1)=0\,|\,\beta( 1)\not =2\}\\
     &\quad+\Prob\{\lambda( 1)=j\,|\lambda(0)=\ell,\beta( 1)=1\}\Prob\{\beta(
     1)=1\,|\,\beta( 1)\not =2\}\\
           &=\Prob\{N( 1)=j\,|\beta( 1)=0\}\Prob\{\beta( 1)=0\,|\,\beta( 1)\not
           =2\}+\Prob\{N( 1)=j\,|\beta( 1)=1\}\\
           &\quad \times\Prob\{\beta( 1)=1\,|\,\beta( 1)\not =2\}=0(1-q)+p_jq=p_jq.
     \end{split}
   \end{equation*}
The derivation of~(\ref{wann}) now follows as in   \cite[Lemma
 2]{quegup11a}  by setting up a recursion on the first passage time of state
 $\ell\in\{1, \dots,\Lambda\}$ to $0$ and then considering the transitions away from
   $0$.
\end{proof}

As a consequence of Lemma~\ref{lem:delta}, $\Omega$ in~(\ref{eq:19}) can be
written as:
\begin{equation*}
  \Omega = \alpha(1-q+p_0q) \big( 1 +\rho \theta^T (I-\rho G)^{-1}e_1 \big),
\end{equation*}
and the  stability condition in
Theorem~\ref{thm:anytime}, $\Omega<1$, becomes
\begin{equation*}
  \begin{bmatrix}
      p_1&\dots&p_{\Lambda}
    \end{bmatrix}
    (I_\Lambda-\rho G)^{-1}e_1<\frac{1-\alpha +\alpha q (1-p_0)}{\alpha \rho q (1-q(1-p_0))},
\end{equation*}
which is independent of the size of   $\mathbb{B}_d$.
\par Sufficient conditions for
stationarity can be  stated as follows:
\begin{thm}[Stationarity with Algorithm A$_1$]
\label{stocAnyTime}
Suppose that
  Assumptions~\ref{ass:CLF} to~\ref{ass:Continuity} hold. If Algorithm A$_1$  is
  used and $\Omega < 1$, then there exists an invariant
  probability measure for $\{x\}_{\K}$ as well as for the
  aggregated Markov process, $\{x_{[k,k-(\Lambda-1)]}\}_{k \in \mathbb{N}}$,
  where  $$x_{[k,k-(\Lambda-1)]} \eq
  \{x(k),x(k-1),\cdots,x(k-\Lambda+1)\}.$$
  Furthermore, under every invariant
  probability measure $\pi$, $\E_{\pi}\{V(x)\} < \varphi_{2}(d)/({1-\Omega}).$\hfs
\end{thm}
\begin{proof}
First note that if $N(k)\geq 1$, then $u(k)$ is determined by the
current state. If the processor is not available, then either  $u(k)$
has been determined by the states which are at most $\Lambda$ time stages old,
or $u(k)=\mathbf{0}_p$. Since the processor availability is independent of the state, the stochastic process $\{x_{[k,k-\Lambda+1]}\}$ is Markovian. Let ${\bf z}(k)\eq x_{[k,k-\Lambda+1]}$. From Assumption
\ref{ass:Continuity}, $\{{\bf z}\}_{\N_0}$ is also weak Feller.

 We first invoke Theorem 2.1 in \cite{yukmey13a} with $\K$ containing the sequence of
 stopping times. Since
\begin{equation}\label{bound0}
\E\{V(x(k_{i+1}))\,|\, x(k_i)=\chi\} \leq V(\chi) -(1-\Omega) V(\chi) + D,\quad  \forall \chi \in\R^n,
\end{equation}
and the sampled chain is weak Feller, it follows that   $\{x\}_{\K}$ admits an invariant probability measure.

Define $\tilde{V}({\bf z}(k)) \eq V(x(k))$. Now, note that by (\ref{eq:21d}), with $\Omega < 1$,
$\E\{\tilde{V}({\bf z}(k_{i+1}))\,|\, {\bf z}(k_i)=\chi\} \leq D+\Omega \tilde{V}(\chi)$, $\forall \chi$.
Thus,
$ \E\{\tilde{V}({\bf z}(k_{i+1}))\,|\, {\bf z}(k_i)=\chi\} \leq \tilde{V}(\chi)
-(1-\Omega) \tilde{V}(\chi) + D$, $\forall \chi$, and since $V$ is monotone
increasing and by Assumption \ref{ass:bound_prob}, there exists a compact set
$\mathcal{S}$ such that for $1 - \Omega > \zeta > 0$, 
$ \E\{V(x(k_{i+1}))\,|\, x(k_i)=x\} \leq V(x) - \zeta V(x) + D1_{x \in
  \mathcal{S}},\quad \forall x \in\R^n.$ 
Since $V(x(t))$ is bounded from below outside $\mathbb{B}_d$, and $x(k) \notin \mathbb{B}_d$ for $k \notin \K$, and that (\ref{driftCriterion2}) implies that for some $M_1 < \infty$
\begin{equation*}
  \E \Bigg\{ \sum_{k=k_i}^{k_{i+1}-1} V(x(k))\,\bigg|\,x(k_0)=\chi\Bigg\} \leq M_1,
\end{equation*}
it follows that $\sup_{x_{k_i}}\E\{k_{i+1}-k_i \,|\, x_{k_i}\} < \infty$. Finally, by Assumption \ref{ass:bound_prob}, if $x_t \in \mathcal{S}$ then $x_{[t+\Lambda-1,t]} \in \bar{\mathcal{S}}$ where $\bar{\mathcal{S}}$ is a compact set. Thus, Theorem 2.2 in \cite{yukmey13a} implies that there exists an invariant probability distribution, $\pi$, for  $\{{\bf z}\}_{\N_0}$.

Since (\ref{bound0}) holds, with $P^m V(\chi) := \E\{V(x(k_m))|x(k_0)=\chi\}$, following arguments similar to the proof of Theorem 2.2 of \cite{yukmey13a}, for every realization of $x(k_0)$, it follows that
\begin{eqnarray*}
(1-\Omega) \limsup_{T \to \infty} {1 \over T} \E \Bigg\{ \sum_{i=0}^{T-1}
V(x(k_i))\Bigg\}
\leq \limsup_{T \to \infty} {1 \over T} \bigg( V(x(k_0)) + \sum_{i=0}^{T-1} D \bigg).
\end{eqnarray*}
Thus, $\limsup_{T \to \infty} (1/T) \sum_{m=0}^{T-1} P^m V(x(k_m)) \leq D/( 1-\Omega)$. Applying Fatou's lemma, we obtain
\begin{eqnarray*}
 \limsup_{T \to \infty} \E_{\pi}\Bigg\{ {1 \over T} \sum_{i=0}^{T-1}
 \min(N,V(x(k_i)))\Bigg\}\leq  \E_{\pi}\Bigg\{\limsup_{T \to \infty} {1 \over
   T}\sum_{i=0}^{T-1} \min(N,V(x(k_i)))\Bigg\} \leq {D \over 1-\Omega}. 
\end{eqnarray*}
Then, by the monotone convergence theorem, by letting $N \to \infty$,
\[\limsup_{T \to \infty} \E_{\pi}\Bigg\{{1 \over T}\sum_{i=0}^{T-1} V(x(k_i))\Bigg\} \leq {D \over 1-\Omega}.\]

Thus, there exists an invariant probability measure both for the original chain and for the sampled
 chain; under every such invariant probability measure $\pi$, $\E_{\pi}\{V(x)\} < D/( 1-\Omega)$.
\end{proof}

\section{Numerical Examples}
\label{sec:numerical-examples}

\begin{figure}[t]
  \begin{center}
      \includegraphics[width=.43\textwidth]{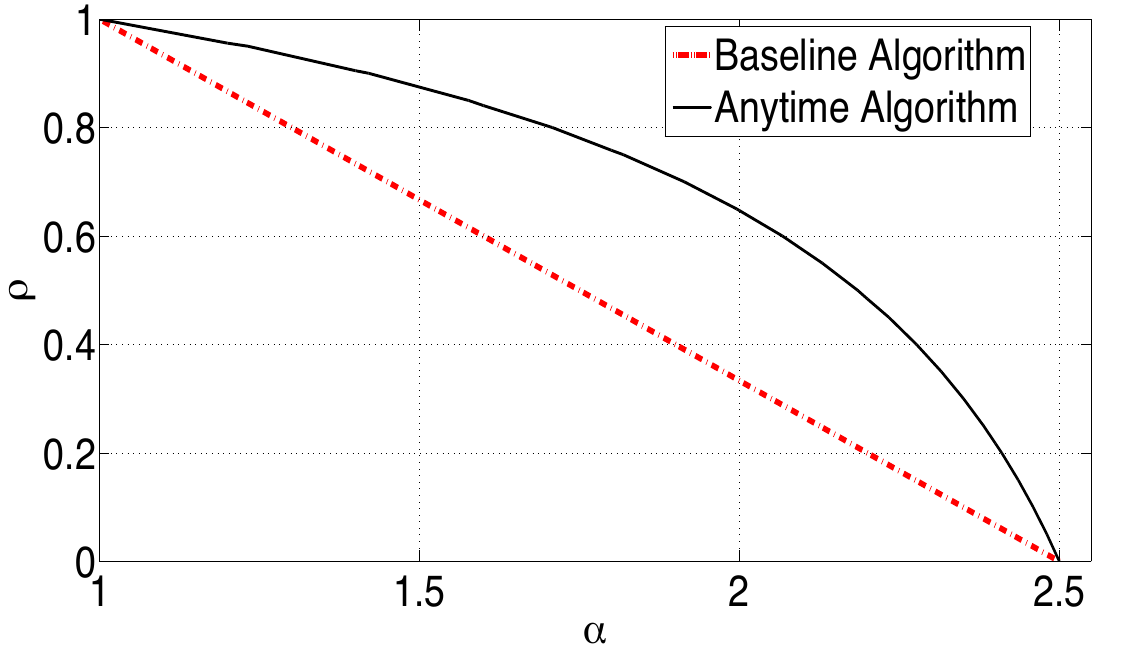}
      \caption{Boundaries of stability: $\Omega =1$ (solid line)  and
        $\Gamma =1$ (dashed).}
\label{fig:boundaries}
\end{center}\vspace{-10mm}
\end{figure}

\begin{figure}[t]
\begin{center}
     \includegraphics[width=.5\textwidth]{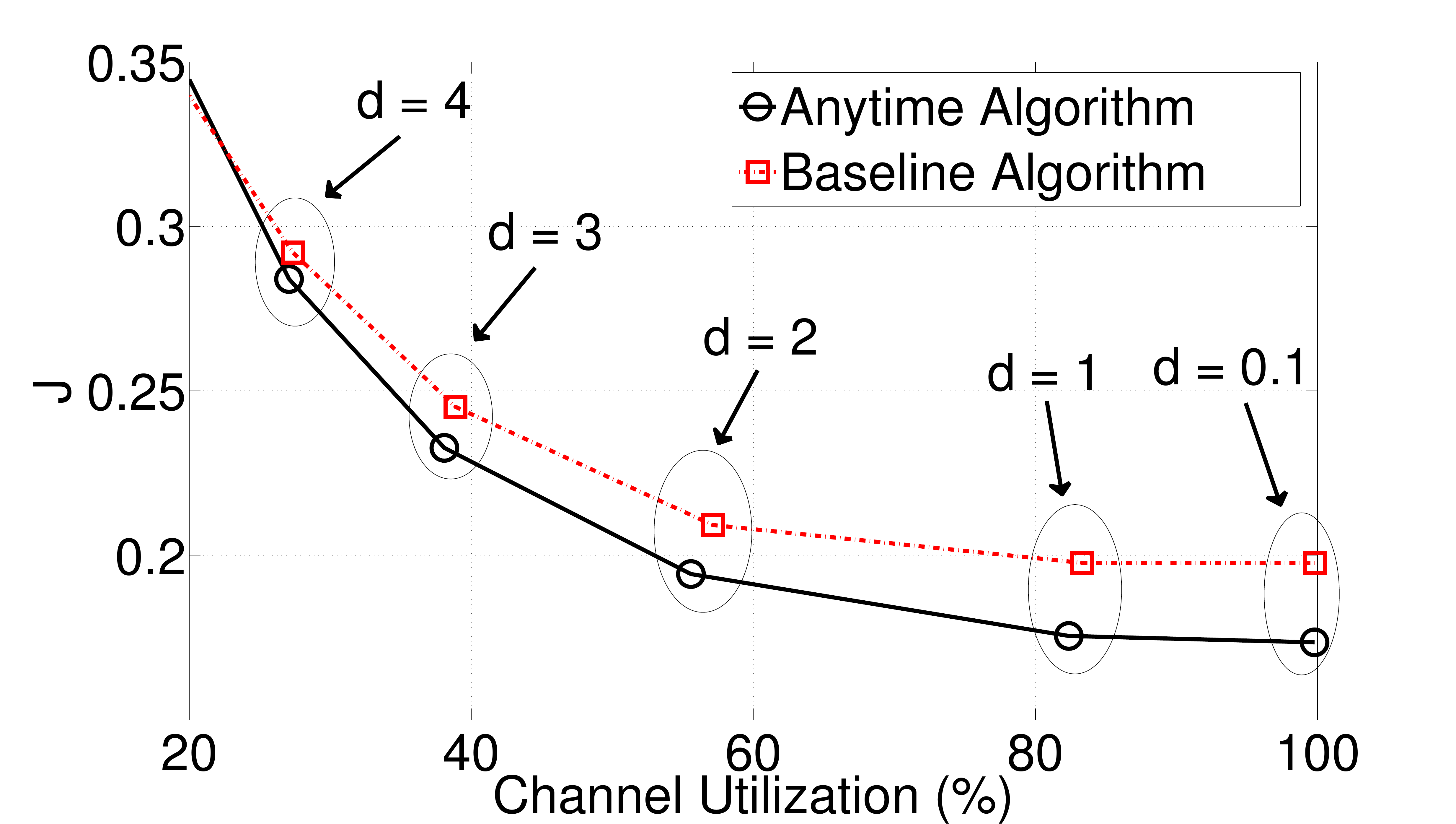}
    \caption{Empirical cost  versus   channel utilization for
      different values of $d$.}
    \label{fig:cost}
 \end{center}\vspace{-8mm}
\end{figure}
We first compare the stability conditions derived for a specific case. Suppose
that the buffer length is  given by
$\Lambda=4$, whereas
$p_i=0.2$, $i\in \{0,\dots,4\}$, and $q=0.75$. The stability region
boundaries, see~(\ref{eq:baseline_stability}) and~(\ref{eq:19}),
in terms of $\alpha$ and $\rho$ are depicted in Fig.~\ref{fig:boundaries}. It
can be seen that the guaranteed stable region (under the curve) provided by our
results is larger when using
Algorithm A$_1$ than when using~(\ref{eq:4}).
\par Next, we consider an open-loop unstable constrained plant model  of the
  form~(\ref{eq:process}), but with additive noise:
\begin{equation*}
 \begin{bmatrix}
  x_1(k+1)\\x_2(k+1)
\end{bmatrix}=
  \begin{bmatrix}
    x_2(k)+u_1(k)\\-{\rm{sat}}(x_1(k)+x_2(k))+u_2(k)
  \end{bmatrix}+
 \begin{bmatrix}
  w_1(k)\\w_2(k)
\end{bmatrix}
\end{equation*}
where
\begin{equation*}
  {\rm{sat}}(\mu)=
  \begin{cases}
    -10,&\text{if $\mu<-10$,}\\
    \mu & \text{if $\mu \in [-10,10]$,}\\
    10,&\text{if $\mu>10$},
  \end{cases}
\end{equation*}
see\cite[Example 2]{quegup13a}. The initial condition $x(0)$ and the disturbance
$w(k)$ are zero-mean i.i.d.\ Gaussian with unit covariance. The control policy
$\kappa$ is taken
as $\kappa(x)=[-x_2\quad 0.505{\rm{sat}}(x_1+x_2)]^T$, $x\in\R^2$. If we choose
$V(x)=2|x|$, then direct calculations give that
\begin{equation*}
 \begin{split}
   V&\big(f(x,\kappa(x)) \big) =
  0.99| {\rm{sat}} (x_1+x_2)| \leq  0.99 | x_1+x_2|\\
  &\leq 1.98\max\{|x_1|,|x_2|\}-\max\{|x_1|,|x_2|\}+|x|\leq 1.98|x|.
\end{split}
\end{equation*}
Thus, Assumption~\ref{ass:CLF} holds with $\rho = 0.99$, and
$\varphi_1(s)=\varphi_2(s)=2s$.  Processor availability and $\Lambda$ are taken as above, but we now set
$q=0.4$. Performance is evaluated through the  empirical cost
$J\triangleq\frac{1}{50}\left(\sum_{k=0}^{49}|x(k)|^2\right)$ and
the Channel Utilization~(\%), calculated as
\begin{equation*}
 \frac{\text{Total number of time steps at which $\beta(k)\neq 2$}}{\text{Total number of time steps}}~ (\%).
\end{equation*}
By averaging over $10^4$
realizations,  Fig.~\ref{fig:cost} is obtained. As can be seen in that figure,
the proposed event-based
anytime control algorithm gives better trade-offs between empirical cost and
channel utilization.

\section{Conclusions}
\label{sec:conclusions}
This work considered the control of a non-linear process with both communication
and processing constraints. A sensor node transmits data to the controller
across a channel that stochastically erases data. The control algorithm is
executed over a processor that can provide only limited, time-varying and a
priori unknown processing resources. To reduce the communication frequency, the
sensor utilizes an event-triggered scheme. Similarly, to better utilize the
processor availability, the control input is calculated by using an anytime
control algorithm. For the resulting system, we present stochastic stability and
stationarity results. Numerical studies illustrate that significant performance
gains can be obtained by using the proposed algorithm. Future work includes the extension of the analysis to noisy systems, and establishing further stability properties such as ergodicity and rates of convergence to equilibrium.



\begin{thebibliography}{10}

\bibitem{gupdan09}
V.~Gupta, A.~F. Dana, J.~P. Hespanha, R.~M. Murray, and B.~Hassibi, ``Data
  transmission over networks for estimation and control,'' {\em {IEEE} Trans.
  Automat. Contr.}, vol.~54, pp.~1807--1819, Aug. 2009.

\bibitem{imeyuk06}
O.~C. Imer, S.~Y\"{u}ksel, and T.~Ba\c{}sar, ``Optimal control of {LTI} systems
  over unreliable communication links,'' {\em Automatica}, vol.~42,
  pp.~1429--1439, Sept. 2006.

\bibitem{schsin07}
L.~Schenato, B.~Sinopoli, M.~Franceschetti, K.~Poolla, and S.~S. Sastry,
  ``Foundations of control and estimation over lossy networks,'' {\em Proc.
  {IEEE}}, vol.~95, pp.~163--187, Jan. 2007.

\bibitem{quenes12a}
D.~E. Quevedo and D.~Ne\v{s}i\'c, ``Robust stability of packetized predictive
  control of nonlinear systems with disturbances and {M}arkovian packet
  losses,'' {\em Automatica}, vol.~48, pp.~1803--1811, Aug. 2012.

\bibitem{lilem10}
L.~Li, M.~Lemmon, and X.~Wang, ``Event-triggered state estimation in vector
  linear processes,'' in {\em Proc.~Amer.~Contr.~Conf.}, pp.~2138--2143, 2010.

\bibitem{tabuad07}
P.~Tabuada, ``Event-triggered real-time scheduling of stabilizing control
  tasks,'' {\em {IEEE} Trans. Automat. Contr.}, vol.~52, pp.~1680--1685, Sept.
  2007.

\bibitem{xuhes04}
Y.~Xu and J.~Hespanha, ``Optimal communication logics in networked control
  systems,'' in {\em Proc.~IEEE Conf.~Decis.~Contr.}, pp.~3527--3532, 2004.

\bibitem{ramsan11b}
C.~Ramesh, H.~Sandberg, and K.~H. Johansson, ``Steady state performance
  analysis of multiple state-based schedulers with {CSMA},'' in {\em Proc.~IEEE
  Conf.~Decis.~Contr.}, 2011.

\bibitem{xiagup13}
M.~Xia, V.~Gupta, and P.~J. Antsaklis, ``Networked state estimation over a
  shared communication medium,'' in {\em Proc.~Amer.~Contr.~Conf.}, 2013.

\bibitem{rabjoh09}
M.~Rabi and K.~H. Johansson, ``Scheduling packets for event-triggered
  control,'' in {\em Proc.~Europ.~Contr.~Conf.}, pp.~3779--3784, 2009.

\bibitem{bliall11}
R.~Blind and F.~Allg\"ower, ``Analysis of networked event-based control with a
  shared communication medium: Part 1 - pure aloha,'' in {\em Proc.~{IFAC}
  World Congr.}, 2011.

\bibitem{cerhen08}
A.~Cervin and T.~Henningsson, ``Scheduling of event-triggered controllers on a
  shared network,'' in {\em Proc.~IEEE Conf.~Decis.~Contr.}, pp.~3601---3606,
  2008.

\bibitem{govfer99}
L.~K. McGovern and E.~Feron, ``Closed-loop stability of systems driven by
  real-time dynamic optimization algorithms,'' in {\em Proc.~IEEE
  Conf.~Decis.~Contr.}, vol.~4, (Phoenix, AZ), pp.~3690--3696, Dec. 1999.

\bibitem{henake04}
D.~Henriksson and J.~{\AA}kesson, ``Flexible implementation of model predictive
  control using sub-optimal solutions,'' Tech. Rep. Internal Report No.
  TFRT-7610-SE, Dep. of Automatic Control, Lund University, 2004.

\bibitem{andseu13a}
P.~Andrianiaina, A.~Seuret, and D.~Simon, ``Robust system control method with
  short execution deadlines.'' European Patent Application EP 2 568 346 A1,
  Airbus Operations Toulouse, March 2013.

\bibitem{cervel10}
A.~Cervin, M.~Velasco, P.~Mart\'\i, and A.~Camacho, ``Optimal online sampling
  period assignment: Theory and experiments,'' {\em {IEEE} Trans. Contr. Syst.
  Technol.}, vol.~18, June 2010.

\bibitem{bhabal04}
R.~Bhattacharya and G.~J. Balas, ``Anytime control algorithms: Model reduction
  approach,'' {\em AIAA Journal of Guidance, Control and Dynamics}, vol.~27,
  pp.~767--776, Sept.--Oct. 2004.

\bibitem{grefon07}
L.~Greco, D.~Fontanelli, and A.~Bicchi, ``Almost sure stability of anytime
  controllers via stochastic scheduling,'' in {\em Proc.~IEEE
  Conf.~Decis.~Contr.}, (New Orleans, LA), pp.~5640--5645, Dec. 2007.

\bibitem{gupluo13}
V.~Gupta and F.~Luo, ``On a control algorithm for time-varying processor
  availability,'' {\em {IEEE} Trans. Automat. Contr.}, vol.~58, Mar. 2013.

\bibitem{quegup13a}
D.~E. Quevedo and V.~Gupta, ``Sequence-based anytime control,'' {\em {IEEE}
  Trans. Automat. Contr.}, vol.~58, pp.~377--390, Feb. 2013.

\bibitem{ramsan11a}
C.~Ramesh, H.~Sandberg, and K.~H. Johansson, ``On the dual effect in
  state-based scheduling of networked control systems,'' in {\em
  Proc.~Amer.~Contr.~Conf.}, pp.~2216--2221, 2011.

\bibitem{quegup11a}
D.~E. Quevedo and V.~Gupta, ``Stability of sequence-based anytime control with
  {M}arkovian processor availability,'' in {\em Proc. Austr. Contr. Conf.},
  2011.

\bibitem{meyn89}
S.~P. Meyn, ``Ergodic theorems for discrete time stochastic systems using a
  stochastic {L}yapunov function,'' {\em SIAM Journal on Control and
  Optimization}, vol.~27, pp.~1409--1439, Nov. 1989.

\bibitem{meytwe09}
S.~Meyn and R.~L. Tweedie, {\em Markov Chains and Stochastic Stability}.
\newblock Cambridge University Press, 2009.

\bibitem{yukmey13a}
S.~Y\"{u}ksel and S.~P. Meyn, ``Random-time, state-dependent stochastic drift
  for {M}arkov chains and application to stochastic stabilization over erasure
  channels,'' {\em {IEEE} Trans. Automat. Contr.}, vol.~58, no.~1, pp.~47--59,
  2013.

\end{thebibliography}
\end{document}